\documentclass{amsart}
\usepackage{graphics,graphicx} 

\usepackage[cp1250]{inputenc}

\usepackage[arrow, matrix, curve]{xy}
\usepackage{amsmath,amsthm,amssymb}
\usepackage{amscd}

\newenvironment{Proof of}[1]{\emph{Proof of #1.}}{$\qquad \square$\par}

\DeclareMathOperator{\dashind}{-Ind}

\DeclareMathOperator{\clsp}{\overline{span}}

\newcommand{\h}{\widehat h}

\newcommand{\hal}{\widehat \alpha}

\newcommand{\al}{\alpha}

\newcommand{\FF}{\mathcal F}
\newcommand{\FFF}{\mathbb F}
\newcommand{\II}{\mathbb I}
\newcommand{\OO}{\mathcal O}
\newcommand{\A}{\mathcal A}
\newcommand{\B}{\mathcal B}
\newcommand{\E}{\mathcal E}

\newcommand{\G}{\mathcal G}
\newcommand{\C}{\mathbb C}

\newcommand{\RR}{\mathcal R}
\newcommand{\Z}{\mathbb Z}
\newcommand{\N}{\mathbb N}
\newcommand{\T}{\mathbb T}

 \newtheorem{thm}{Theorem}
 
 \newtheorem{lem}[thm]{Lemma}
 \newtheorem{prop}[thm]{Proposition}
 \theoremstyle{definition}
 
 \theoremstyle{remark}

\begin{document}
\title[Uniqueness property for circle action
$C^*$-algebras]{Uniqueness property for $C^*$-algebras given by
  relations with circular symmetry}

\author{B. K.  Kwa\'sniewski} 
 \address{Institute of Mathematics\\ University  of Bia{\l}ystok\\
 ul.~Akademicka 2\\ PL-15-267  Bia{\l}ystok\\ Poland}
 \email{bartoszk@math.uwb.edu.pl}
 \thanks{The work  in part supported by  National Science Centre  grants numbers  DEC-2011/01/D/ST1/04112, DEC-2011/01/B/ST1/03838.}
\begin{abstract}
  A general method of investigation of the uniqueness property for
  $C^*$-algebra equipped with a circle gauge action is discussed. It
  unifies isomorphism theorems for various crossed products and
  Cuntz-Krieger uniqueness theorem for Cuntz-Krieger algebras.
\end{abstract}
\keywords{uniqueness property, topological freeness, Hilbert bimodule,
  crossed product, Cuntz-Krieger algebra}

\subjclass{46L05,46L55}
   
\maketitle

\section{Introduction}

The origins of $C^*$-theory and particularly the theory of universal
$C^*$-algebras generated by operators that satisfy prescribed
relations go back to the work of W. Heisenberg, M. Bohr and P. Jordan
on matrix formulation of quantum mechanics, and among the most
stimulating examples are algebras generated by anti-commutation
relations and canonical commutation relations (in the Weyl form). The
great advantage of relations of CAR and CCR type is \emph{uniqueness
  of representation}. Namely, due to the celebrated Slawny's theorem,
see e.g. \cite{EL}, the $C^*$-algebras generated by such relations are
defined uniquely up to isomorphisms preserving the generators and
relations. This \emph{uniqueness property} is not only a strong
mathematical tool but also has a significant physical meaning -- if we
had no such uniqueness, \emph{different representations would yield
  different physics}.

The aim of the present note is to advertise a program of developing a
general approach to investigation of uniqueness property and related
problems based on exploring the symmetries of relations.  We focus
here, as a first attempt, on circular symmetries and propose a
two-step method of investigation universal $C^*$-algebra $C^*(\G,\RR)$
generated by a set of generators $\G$ subject to relations $\RR$ which
could be schematically presented as follows:
$$ \xymatrix{\underset{\text{relations,   circle
      action}}{(\mathcal{G},\mathcal{R}, \{\gamma_\lambda\}_{\lambda
      \in \T})} \ar[r]^{\text{step }1} &\underset{ \substack{\text{Hilbert
      bimodule}\\ \text{(reversible dynamics)}}}{(\B_0,\B_1)}
  \ar[r]^{\text{step }2\,\,\,\,}  &\underset{ \text{universal
    }C^*\text{-algebra}}{C^*(\mathcal{G},\mathcal{R})=\B_0\rtimes_{\B_1}
    \Z} 
}
$$
-- we fix a circle gauge action $\gamma=\{\gamma\}_{\lambda \in \T}$
on $C^*(\mathcal{G},\mathcal{R})$ which is induced by a circular
symmetry in $(\mathcal{G},\mathcal{R})$; in the first step we
associate to $\gamma$ a non-commutative reversible dynamical system
which is realized via a Hilbert bimodule $(\B_0,\B_1)$, and in the
second  step we use this system to determine the uniqueness property
for $C^*(\mathcal{G},\mathcal{R})$.

\section{Uniqueness property, universal $C^*$-algebras and gauge
  actions}

Suppose we are given an abstract set of generators $\G$ and a set of
$^*$-algebraic relations $\RR$ that we want to impose on
$\G$. Formally $\G$ is a set and $\RR$ is a set consisting of certain
$^*$-algebraic relations in a free non-unital $^*$-algebra $\FFF$
generated by $\G$.  By a \emph{representation} of the pair $(\G,\RR)$
we mean a set of bounded operators $ \pi=\{\pi(g)\}_{g\in \G}\subset
L(H)$ on a Hilbert space $H$ satisfying the relations $\RR$, and
denote by $C^*(\pi)$ the $C^*$-algebra generated by $\pi(g)$, $g\in
\G$.  At this very beginning one faces the following two fundamental
problems:
\begin{itemize}
\item[ 1.]  (\emph{non-degeneracy problem}) Do there exists a
  \emph{faithful representation} of $(\G,\RR)$, i.e. a representation
  $\{\pi(g)\}_{g\in \G}$ of $(\G,\RR)$ such that $\pi(g)\neq 0$ for
  all $g\in \G$?
\item[2.]  (\emph{uniqueness problem}) If one has two different
  faithful representation of $(\G,\RR)$, do they generate isomorphic
  $C^*$-algebras? More precisely, does for any two faithful
  representations $\pi_1$, $\pi_2$ of $(\G,\RR)$ the mapping
$$
\pi_1(g) \longmapsto \pi_2(g), \qquad g\in \G,
$$
extends to the (necessarily unique) isomorphism $C^*(\pi_1)\cong
C^*(\pi_2)$?
  
\end{itemize}
The first problem is important and interesting in its own rights, see
\cite{kwa-leb1}, \cite{kwa-doplicher}, however here  we would like  to focus on the second problem and thus 
throughout we   assume that all the pairs $(\G,\RR)$ under consideration are
non-degenerate. We say that
$(\G,\RR)$ possess \emph{uniqueness property} if the answer to
question 2 is positive.

Any representation $\pi$ of $(\G,\RR)$ extends uniquely to a
$^*$-homomorphism, also denoted by $\pi$, from $\FFF$ into $L(H)$.
The pair $(\G,\RR)$ is said to be \emph{admissible} if the function
$|||\cdot|||:\FFF\to [0,\infty]$ given by
$$
|||w|||=\sup\{\|\pi(w)\|: \pi \text{ is a representation of } (\G,\RR)\}
$$
is finite. Plainly, admissibility is a necessary condition for
uniqueness property and therefore we make it our another standing
assumption. Then the function $|||\cdot|||:\FFF\to [0,\infty)$ is a
$C^*$-seminorm on $\FFF$ and its kernel
$$
\II:=\{w\in \FFF: |||w|||=0\}
$$
is a self-adjoint ideal in $\FFF$ -- it is the smallest self-adjoint
ideal in $\FFF$ such that the relations $\RR$ become valid in the
quotient $\FFF/\II$.  We
put
$$
 C^*(\G,\RR) := \overline{\FFF/\II}^{|||\cdot|||}
$$ 
and call it a \emph{universal $C^*$-algebra} generated by $\G$ subject
to relations $\RR$, cf. \cite{blackadar}. $C^*$-algebra $C^*(\G,\RR)$
is characterized by the property that any representation of $(\G,\RR)$
extends uniquely to a representation of $C^*(\G,\RR)$ and all
representations of $C^*(\G,\RR)$ arise in that manner.  In particular,
$(\G,\RR)$ possess uniqueness property if and only if any faithful
representation of $(\G,\RR)$ extends to a faithful representation of
$C^*(\G,\RR)$.

\section{Gauge actions -- exploring the symmetries in the relations}
We would like to identify the uniqueness property of $(\G,\RR)$ by
looking at the symmetries in $(\G,\RR)$. In order to formalize this we use
a natural torus action $\{\gamma_{\lambda}\}_{\lambda\in\T^\G}$ on
$\FFF$ determined by the formula
$$
\gamma_\lambda(g)=\lambda_g\, g,\qquad \text{ for }\,\, g\in \G
\,\,\text{ and }\,\, \lambda=\{\lambda_h\}_{h\in \G}\in \T^\G
$$
where $\T=\{z\in \C:|z|=1\}$ is a unit circle.  A closed subgroup
$G\subset \T^\G$ may be considered as a \emph{group of symmetries in
  the pair $(\G,\RR)$} if the restricted action
$\gamma=\{\gamma_{\lambda}\}_{\lambda\in G}$ leaves invariant the
ideal $\II$. Any such group gives rise to a point-wisely continuous
group action on $C^*(\G,\RR)$ and actions that arise in that manner
are called \emph{gauge actions}.
\\
Let us from now on consider the case when $G\cong \T$, that is we have
a circle gauge action $\gamma=\{\gamma_{\lambda}\}_{\lambda\in \T}$ on
$C^*(\G,\RR)$. Then for  each $n\in \Z$ the formula
$$
\E_n(b) :=\int_{\T}\gamma_\lambda(b)\lambda^{-n} \, d\lambda  
$$ 
defines a projection $\E_n:C^*(\G,\RR)\to C^*(\G,\RR)$, called
\emph{$n$-th spectral projection}, onto the subspace
$$
\B_n:=\{b\in C^*(\G,\RR): \gamma_\lambda(b)=\lambda^n b\} 
$$
called \emph{$n$-th spectral subspace} for $\gamma$,
cf. e.g. \cite{exel1}. Spectral subspaces specify a $\Z$-gradation on
$C^*(\G,\RR)$. Namely, $\bigoplus_{n\in\Z}\B_n$ is dense in
$C^*(\G,\RR)$, and
\begin{equation}\label{z-gradation relations}
\B_n\B_m\subset \B_{n+m},\,\,\, \B_n^*=\B_{-n} \,\,\text{ for all }\,\, n,m \in \Z.
\end{equation} 
In particular, $\B_0$ is a $C^*$-algebra -- the fixed point algebra
for $\gamma$, and $\E_0:\B\to \B_0$ is a conditional expectation.  A
circle action on a $C^*$-algebra $\B$ is called \emph{semi-saturated}
\cite{exel1} if $\B$ is generated as a $C^*$-algebra by its first and
zeroth spectral subspaces. We note that  every continuous group endomorphism
of $\T$ is of the form $\lambda\mapsto \lambda^n$, for certain
$n\in\Z$, and hence  it follows that $\G\subset \cup_{n\in \Z}\B_n$. In
particular, we have

\begin{lem} The circle gauge action
  $\gamma=\{\gamma_{\lambda}\}_{\lambda\in \T}$ on $C^*(\G,\RR)$ is
  semi-saturated, that is $C^*(\G,\RR)=C^*(\B_0,\B_1)$ if and only if
  $\G=\G_0\cup \G_1$ for some disjoint sets $\G_0$, $\G_1$ and $
  \gamma_\lambda(g_0)= g_0$, $\gamma_\lambda(g_1)= \lambda g_1$, for
  all $g_i\in \G_i$.
\end{lem}
We introduce an important necessary condition for $(\G,\RR)$ to
possess uniqueness property.
\begin{prop}\label{gauge uniqueness theorem} The following conditions
  are equivalent:
  \begin{itemize}
  \item[i)] each faithful representation of $(\G,\RR)$ give rise to a
    faithful representation of the fixed-point algebra $\B_0$.
  \item[ii)] each faithful representation $\pi$ of $(\G,\RR)$ give
    rise to a faithful representation of $C^*(\G,\RR)$ if and only if
    there is a circle action $\beta$ on $C^*(\pi)$ such that
$$
\beta_z(\pi(g))=\pi(\gamma_z(g)), \qquad g\in \G.
$$
\end{itemize}
\end{prop}
\begin{proof}
  i) $\Longrightarrow$ ii). It suffices to apply the gauge invariance
  uniqueness for circle actions, see e.g.  \cite[2.9]{exel1} or
  \cite[4.2]{dr}. ii) $\Longrightarrow$ i). Assume that $\pi$ is a
  faithful representation of $(\G,\RR)$ such that its extension is not
  faithful on $\B_0$.  The spaces $\{\pi(\B_n)\}_{n\in \Z}$ form a
  $\Z$-graded $C^*$-algebra and thus by \cite[4.2]{dr}, there is a
  (unique) $C^*$-norm $\|\cdot\|_\beta$ on $\bigoplus_{n\in
    \Z}\pi(\B_n)$ such that the circle action $\beta$ on
  $\bigoplus_{n\in \Z}\pi(\B_n)$ established by gradation extends onto
  the $C^*$-algebra $\B=\overline{\bigoplus_{n\in
      \Z}\pi(\B_n)}^{||\cdot||_\beta}$. Composing $\pi$ with the
  embedding $\bigoplus_{n\in \Z}\pi(\B_n)\subset \B$ one gets a
  faithful representation $\pi'$ of $(\G,\RR)$ which is
  gauge-invariant but not faithful on $C^*(\G,\RR)$.
\end{proof}
In the literature the statements showing that the condition ii) in
Proposition \ref{gauge uniqueness theorem} holds are often
called \emph{gauge-invariance uniqueness theorems} and therefore we shall
say that the triple $(\G,\RR, \gamma)$ has the \emph{gauge-invariance
  uniqueness property} if each faithful representation of $(\G,\RR)$
give rise to a faithful representation of the fixed-point algebra
$\B_0$. In particular, this always holds for triples $(\G,\RR,
\gamma)$ such that $C^*(\G,\RR)$ can be 
modeled as relative Cuntz-Pimsner algebra, see \cite[Sect. 9]{kwa-doplicher}
and sources cited there.

\section{From relations to Hilbert bimodules}
Let us fix a pair $(\G,\RR)$ with a circle gauge action
$\gamma=\{\gamma_{\lambda}\}_{\lambda\in \T}$.
It follows from \eqref{z-gradation relations} that $\B_1$ can be
naturally viewed as a \emph{Hilbert bimodule} over $\B_0$, in the
sense introduced in \cite[1.8]{BMS}.  Namely, $\B_1$ is a
$\B_0$-bimodule with bimodule operations inherited from $C^*(\G,\RR)$
and additionally is equipped with two $\B_0$-valued inner products
$$
 \langle a,b\rangle_{R}:=a^*b, \qquad  {_{L}\langle a,b\rangle}:=ab^*  
$$
that satisfy the so-called imprimitivtiy condition: $ a\cdot \langle
b,c\rangle_{R}= {_{L}\langle a,b\rangle}\cdot c = ab^*c, $ for all
$a,b,c\in \B_1$. %
Thus we can consider \emph{crossed product} $\B_1\rtimes_{\B_0} \Z$ of
$\B_0$ by \emph{the Hilbert bimodule} $\B_1$ constructed in
\cite{aee}, which could be alternatively defined as the universal
$C^*$-algebra:
 $$
 \B_1\rtimes_{\B_0} \Z =C^*(\G_\gamma, \RR_\gamma)
 $$
 where $\G_\gamma=\B_0\cup \B_1$ and $\RR_{\gamma}$ consists of all
 algebraic relations in the Hilbert bimodule $(\B_0,\B_1)$.
 \begin{prop}\label{takie tam prr} We have a natural embedding $
   \B_1\rtimes_{\B_0} \Z\hookrightarrow C^*(\G,\RR) $ which is an
   isomorphism if and only if $\gamma$ is semi-saturated.  Moreover,
   if $\gamma$ is semi-saturated, then the following conditions are
   equivalent:
   \begin{itemize}
   \item[i)] $(\G,\RR)$ possess uniqueness property
   \item[ii)] $(\G,\RR,\gamma)$ has gauge-invariance uniqueness
     property and $(\G_\gamma,\RR_\gamma)$ possess uniqueness property
   \end{itemize}
 \end{prop}
 \begin{proof}
   Since the homomorphism $ \B_1\rtimes_{\B_0} \Z\mapsto C^*(\G,\RR) $
   is gauge-invariant and injective on $\B_0$ it is injective onto the
   whole $ \B_1\rtimes_{\B_0} \Z$ by \cite[2.9]{exel1}. The rest, in view of
   Proposition \ref{gauge uniqueness theorem}, is clear.
\end{proof}
The Hilbert bimodule $(\B_0,\B_1)$ is an imprimitivity bimodule
(called also Morita-Rieffel equivalence bimodule), see \cite{morita},
if and only if $\overline{\B_1^*\B_1}= \B_0$ and
$\overline{\B_1\B_1^*}=\B_0$. In general, $\overline{\B_1^*\B_1}$ and
$\overline{\B_1\B_1^*}$ are non-trivial ideals in $\B_0$ and we may
treat $\B_1$ as a
$\overline{\B_1\B_1^*}-\overline{\B_1^*\B_1}$-imprimitivity
bimodule. This means, cf. \cite[Cor. 3.33]{morita}, that the induced
representation functor
$$
 \h= \B_1\dashind
$$
is a homeomorphism $\h:\overline{\B_1^*\B_1}\to \overline{\B_1\B_1^*}$
between the spectra of $\overline{\B_1^*\B_1}$ and
$\overline{\B_1\B_1^*}$. Treating these spectra %
as open subsets of the spectrum $\widehat{\B}_0$ of $\B_0$ we may
treat $\h$ as a partial homeomorphism of $\widehat{\B}_0$.  We shall
say that $(\widehat{\B},\h)$ is a \emph{partial dynamical system dual
  to the bimodule} $(\B_0,\B_1)$. Partial homeomorphism $\h$ is said
to be \emph{topologically free} if for each $n\in N$ the set of points
in $\widehat{\B_0}$ for which the equality $\h^n(x)=x$ (makes sense
and) holds has empty interior.
\begin{thm}[main result]\label{main result}
  Suppose that the partial homeomorphism $\h=B_1\dashind$ is
  topologically free.  Then the pair $(\G_\gamma,\RR_\gamma)$ possess
  uniqueness property and in particular
\begin{itemize}
\item[i)] if $(\G,\RR,\gamma)$ possess gauge-invariance uniqueness
  property, then any faithful representation of $(\G,\RR)$ extends to
  the faithful representation of $ \B_1\rtimes_{\B_0} \Z\subset
  C^*(\G,\RR)$.
\item[ii)] if $\gamma$ is semi-saturated and $(\G,\RR,\gamma)$ possess
  gauge-invariance uniqueness property, then $(\G,\RR)$ possess
  uniqueness property.
\end{itemize}
\end{thm}
\begin{proof}
  Apply the main result of \cite{kwa} and Proposition \ref{takie tam
    prr}.
\end{proof}

\section{Applications to crossed products and Cuntz-Krieger algebras}

We show that our main result is a generalization of the so called
isomorphisms theorem for crossed products by automorphisms (see, for
instance, \cite[pp. 225, 226]{Anton_Lebed} for a brief survey of such
results) by applying it to a  crossed  product by an endomorphism which is
considered to be one of the most successful constructions of this
sort, see \cite{Ant-Bakht-Leb} and sources cited there. In particular,
we shall use this crossed product to identify the uniqueness property
for Cuntz-Krieger algebras.

\subsection{Crossed products by monomorphisms with hereditary range} %
Let $\al:\A\to\A$ be a monomorphism of a unital $C^*$-algebra $\A$.
Let $\G=\A\cup \{S\}$ and let $\RR$ consists of all $^*$-algebraic
relations in $\A$ plus the covariance relations
\begin{equation}\label{covariance conditions for monomorphism}
  S a S^*= \al(a), \qquad S^* S =1, \qquad  a\in \A.
\end{equation}
Then $C^*(\G,\RR)\cong \A\rtimes_\al \N$ is the crossed product of
$\A$ by $\al$, which is equipped with a semi-saturated circle gauge
action: $ \gamma_\lambda(a) =a$, $\gamma_\lambda(S)=\lambda S$, $a\in
\A$.  Let us additionally assume that $\al(\A)$ is a hereditary
subalgebra of $\A$. This is equivalent to $\al(\A)=\al(1)\A\al(1)$.
Then  we have $S^*\A S\subset \A$  since for any $a\in \A$
there is $b\in \A$ such that $\al(b)=\al(1)a \al(1)$ and therefore
$$
  S^* aS =S^*\al(1)a \al(1) S= S^* \al(b) S=S^*S b S^*S=b\in \A.
$$ 
Hence on one hand $\A=\B_0$ is the fixed point algebra for $\gamma$ and $\B_1=
\B_0S$ is the first spectral subspace.  On the other hand the spectrum
of the hereditary subalgebra $\al(\A)$ may be naturally identified
with an open subset of $\widehat{\A}$, see e.g.
\cite[Thm~5.5.5]{Murphy}, and then the dual
$\widehat{\al}:\widehat{\al(\A)}\to \widehat{\A}$ to the isomorphism
$\al:\A\to \al(\A)$ becomes a partial homeomorphism of
$\widehat{\A}$. Under this identification one gets
$$
 \widehat{\al}=\B_1\dashind
$$
and hence if the partial system $(\widehat{\A},\widehat{\al})$ dual to
$(\A,\al)$ is topologically free, then $(\G,\RR)$ possess uniqueness
property.

\subsection{Cuntz-Krieger algebras} 
Let $\G=\{S_i: i=1,...,n\}$, where $n\geq 2$, and let $\RR$ consists
of the Cuntz-Krieger relations
\begin{equation}\label{Cuntz-Krieger algebras}
  S_i^*S_i =\sum_{j=1}^n A(i,j) S_jS_j^*,  \qquad S_i^* S_k=
  \delta_{i,k} S_i^*S_i, \qquad i,k=1 ,...,n,
\end{equation} 
where $\{A(i,j)\}$ is a given $n\times n$ zero-one matrix such that
every row and every column of $A$ is non-zero, and $\delta_{i,j}$ is
Kronecker symbol.  Then $C^*(\G,\RR)$ is the Cuntz-Krieger algebra
$\OO_A$ and the celebrated Cuntz-Krieger uniqueness theorem,
cf. \cite[Thm. 2.13]{CK}, states that the pair $(\G,\RR)$ possess
uniqueness property if and only if the so called \emph{condition (I)}
holds:
$$
\textbf{(I) } \,\,\,\,\, \begin{array}{l} \text{the space }
  X_A:=\{(x_1,x_2,....)\in\{1,...,n\}^\N: A(x_k,x_{k+1})=1\} \\
  \text{has no isolated points (considered with the product
    topology)}\end{array}
$$
We may recover this result applying our method to the standard circle
gauge action on $\OO_A$ determined by equations $ \gamma_\lambda(S_i)
=\lambda S_i$, $i=1,...,n$. Indeed, the fixed point $C^*$-algebra for
$\gamma$ coincides with the so called AF-core
$$
 \FF_A=\clsp\{S_\mu S^*_\nu : \vert\mu\vert=\vert\nu\vert =k, \, k=\,1,\,\dotsc\}
$$
where for a multiindex $\mu = (i_1,\dots,i_k)$, with $i_j \in 1,...,n$,
we denote by $\vert \mu \vert$ the length $k$ of $\mu$ and write
\,$S_\mu = S_{i_1}S_{i_2}\dotsm S_{i_k}$. Moreover, any faithful
representation of the Cuntz-Krieger relations \eqref{Cuntz-Krieger
  algebras} yields a faithful representation of $\FF_A$, that is
$(\G,\RR,\gamma)$ possess gauge-invariance uniqueness
property. Following \cite{Ant-Bakht-Leb} we put
$$
S:=\sum_{i,j}\frac{1}{ \sqrt{n_j}}\,\, S_i P_j
$$
where $n_j=\sum_{i=1}^n A(i,j)$ and $P_j=S_jS_j^*$,
$j=1,...,n$. 
A routine computation shows that $S\FF_A S^*\subset \FF_A$, $S^*\FF_A
S\subset \FF_A$ and $S^*S=1$ ($S$ is an isometry). Hence the mapping
$\FF_A\ni a \mapsto \al(a):=SaS^* \in \FF_A$ is a monomorphism with a
hereditary range. It 
is uniquely determined by the formula
\begin{equation}\label{endomorphism action on F_A}
  \al \Big(S_{i_2\mu} S^*_{j_2\nu}\Big)= \frac{1}{
    \sqrt{n_{i_2}n_{j_2}}} \sum_{i,j=1}^n S_{i\,i_2\mu}
  S^*_{j\,j_2\nu}.
\end{equation}
From the construction any representation of relations
\eqref{Cuntz-Krieger algebras} yields a representation of
$(\FF_A,\al)$ as introduced in the previous subsection. Conversely, if
$S$ satisfies \eqref{covariance conditions for monomorphism} where
$\A=\FF_A$, then one gets representation of \eqref{Cuntz-Krieger
  algebras} by putting $S_i:= \sum_{j=1}^n A(i,j)\sqrt{n_j} P_i S
P_j$. Thus we have a natural isomorphism
$$
\OO_A\cong \FF_A \rtimes_\al \N
$$ 
under which the introduced gauge actions coincide. Hence we may
identify the partial dynamical system dual to the Hilbert bimodule
$(\B_1, \B_0)$ where $\B_0=\FF_A$ and $\B_1=\FF_AS$ with the partial
dynamical system $(\widehat{\FF_A},\hal)$ dual to $(\FF_A,\al)$, as
introduced in the previous subsection.
\\
In order to identify the the topological freeness of $\hal$ we define
$\pi_\mu\in \widehat{\A}$ for any infinite path $\mu=(i_1,i_2,\dots)$,
$A(i_j,i_{j+1})=1$, $j\in\N$, to be the the GNS-representation
associated to the pure state $\omega_\mu:\FF_A\to \C$ determined by
the formula
\begin{equation}\label{product states equation}
  \omega_\mu(S_\nu S_\eta^*)=
\begin{cases}
  1 & \nu=\eta=(\mu_1,...,\mu_n) \\
  0 & \text{otherwise}
\end{cases}\qquad \text{for}\quad \vert\nu\vert=\vert\eta\vert =n.
\end{equation}
Using description of the ideal structure in $\FF_A$ in terms of
Bratteli diagrams \cite{Bratteli}, similarly as in \cite{kwa}, one can
show that representations $\pi_\mu$ form a dense subset of
$\widehat{\FF_A}$ and
$$
\widehat{\al}(\pi_{(\mu_1,\mu_{2},\mu_3,...)})=\pi_{(\mu_{2},\mu_3,...)},\qquad\text{for
  any }(\mu_1,\mu_{2},\mu_3,...).
$$
In particular, it follows that \emph{topological freeness of $\hal$ is
  equivalent to condition $(I)$}.  Accordingly
\begin{quote}
  \emph{our main result, Theorem \ref{main result}, when applied to
    Cuntz-Krieger relations is equivalent to the Cuntz-Krieger
    uniqueness theorem.}
\end{quote}

\end{document}